\documentclass[12pt,twoside]{article}
\usepackage[centertags]{amsmath}
\usepackage{amsfonts}
\usepackage{amssymb}
\usepackage{amsthm}
\usepackage{newlfont}

\date{}
\def\NL{\hfill\break}
\def\ni{\noindent}
\newcommand{\set}[1]{\left\{#1\right\}}
\newcommand{\parth}[1]{\left(#1\right)}

\begin{document}
%\centerline{\bf International Mathematical Forum, Vol. 6, 2011, no. 35, 1739 - 1747}
%
%\centerline{}
%
%\centerline{}

\centerline {\Large{\bf On Skew Polynomials over p.q.-Baer}}

\centerline{}

\centerline {\Large{\bf and p.p.-Modules}}

\centerline{}

\centerline{}

\centerline{\bf {Mohamed Louzari}}

\centerline{}

\centerline{Department of mathematics}

\centerline{Abdelmalek Essaadi University}

\centerline{B.P. 2121 Tetouan, Morocco}

\centerline{mlouzari@yahoo.com}
\newtheorem{Theorem}{\quad Theorem}[section]
\newtheorem{Definition}[Theorem]{\quad Definition}
\newtheorem{Proposition}[Theorem]{\quad Proposition}
\newtheorem{Corollary}[Theorem]{\quad Corollary}
\newtheorem{Lemma}[Theorem]{\quad Lemma}
\newtheorem{Example}[Theorem]{\quad Example}
\newtheorem{Remark}[Theorem]{\quad Remark}
\renewcommand{\thefootnote}{\fnsymbol{footnote}}
\renewcommand{\thefootnote}{\fnsymbol{footnote}}

\centerline{}

\centerline{}

\centerline{\small This work is dedicated to my Professor El Amin Kaidi Lhachmi from University}

\centerline{\small of Almer\'ia on the occasion of his 62nd birthday.}

\centerline{}

\begin{abstract}Let $M_R$ be a module and $\sigma$ an endomorphism of $R$. Let $m\in M$ and $a\in R$, we say that $M_R$ satisfies the condition $\mathcal{C}_1$ (respectively, $\mathcal{C}_2$), if $ma=0$ implies
$m\sigma(a)=0$ (respectively, $m\sigma(a)=0$ implies $ma=0$). We show that if $M_R$ is p.q.-Baer then so is $M[x;\sigma]_{R[x;\sigma]}$ whenever $M_R$ satisfies the condition $\mathcal{C}_2$, and the converse holds when $M_R$ satisfies the condition $\mathcal{C}_1$. Also, if $M_R$ satisfies $\mathcal{C}_2$ and $\sigma$-skew Armendariz, then $M_R$ is a p.p.-module if and only if $M[x;\sigma]_{R[x;\sigma]}$ is a p.p.-module if and only if $M[x,x^{-1};\sigma]_{R[x,x^{-1};\sigma]}$ ($\sigma\in Aut(R)$) is a p.p.-module. Many generalizations are obtained and more results are found when $M_R$ is a semicommutative module.

\end{abstract}



\footnote[0]{Published in Inter. Math. Forum, Vol. 6, 2011, no. 35, 1739 - 1747}

\section{Introduction}

In this paper, $R$ denotes an associative ring with unity and modules are unitary. We write $M_R$ to mean that $M$ is
a right module. Throughout, $\sigma$ is an endomorphism of $R$ (unless specified otherwise), that is,
$\sigma\colon R\rightarrow R$ is a ring homomorphism with $\sigma(1)=1$. The set of all endomorphisms (respectively, automorphisms)
of $R$ is denoted by $End(R)$ (respectively, Aut(R)). In \cite{Kaplansky}, Kaplansky introduced Baer rings as rings in which the right
(left) annihilator of every nonempty subset is generated by an idempotent. According to Clark \cite{clark}, a ring $R$
is said to be {\it quasi-Baer} if the right annihilator of each right ideal of  $R$ is generated (as a right ideal) by
an idempotent. These definitions are left-right symmetric. Recently, Birkenmeier et al. \cite{birk/pqBaer} called a ring
$R$ a {\it right} $($respectively, {\it left$)$ principally quasi-Baer} (or simply {\it right} $($respectively, {\it left$)$ p.q.-Baer}) if the
right (respectively, left) annihilator of a principally right (respectively, left) ideal of $R$ is generated by an idempotent. $R$ is
called a {\it p.q.-Baer} ring if it is both right and left p.q.-Baer. A ring $R$ is a right (respectively, left) {\it p.p.-ring} if
the right (respectively, left) annihilator of an element of $R$ is generated by an idempotent. $R$ is called a {\it p.p.-ring}
if it is both right and left p.p.-ring.

\smallskip

Lee-Zhou \cite{lee/zhou} introduced Baer, quasi-Baer and p.p.-modules as follows:
\NL$(1)$ $M_R$ is called {\it Baer} if, for any subset $X$ of $M$, $r_R(X)=eR$ where $e^2=e\in R$.
\NL$(2)$ $M_R$ is called {\it quasi-Baer} if, for any submodule $N$ of $M$, $r_R(N)=eR$ where $e^2=e\in R$.
\NL$(3)$ $M_R$ is called {\it p.p.} if, for any $m\in M$, $r_R(m)=eR$ where $e^2=e\in R$.

\smallskip

In \cite{baser2007}, a module $M_R$ is called {\it principally quasi Baer} (p.q.-Baer for short) if, for any $m\in M$,
$r_R(mR)=eR$ where $e^2=e\in R$. It is clear that $R$ is a right p.q.-Baer ring if and only if $R_R$ is a p.q.-Baer module. If $R$
is a p.q.-Baer ring, then for any right ideal $I$ of $R$, $I_R$ is a p.q.-Baer module. Every submodule of a p.q.-Baer
module is p.q.-Baer module. Moreover, every quasi-Baer module is p.q.-Baer, and every Baer module is quasi-Baer module.

\par A ring $R$ is called {\it semicommutative} if for every $a\in R$, $r_R(a)$ is an ideal of $R$ (equivalently, for
any $a,b\in R$, $ab=0$ implies $aRb=0$). In \cite{rege2002}, a module $M_R$ is semicommutative, if for any $m\in M$ and $a\in R$, $ma=0$ implies $mRa=0$. Let $\sigma$ an endomorphism of $R$, $M_R$ is called $\sigma$-semicommutative module \cite{zhang/chen} if, for any $m\in M$ and $a\in R$, $ma=0$ implies $mR\sigma(a)=0$. According to Annin \cite{annin}, a module $M_R$ is $\sigma$-{\it compatible}, if for any $m\in M$ and $a\in R$, $ma=0$ if and only if $m\sigma(a)=0$.

In \cite{lee/zhou}, Lee-Zhou introduced the following notations.  For a module $M_R$, we consider
\par $M[x;\sigma]:=\set{\sum_{i=0}^sm_ix^i:s\geq 0,m_i\in M},$
\par $M[[x;\sigma]]:=\set{\sum_{i=0}^\infty m_ix^i:m_i\in M},$
\par $M[x,x^{-1};\sigma]:=\set{\sum_{i=-s}^tm_ix^i:\;t\geq 0,s\geq 0,m_i\in M},$
\par $M[[x,x^{-1};\sigma]]:=\set{\sum_{i=-s}^\infty m_ix^i:s\geq 0,m_i\in M}.$

\NL Each of these is an Abelian group under an obvious addition operation. Moreover  $M[x;\sigma]$ becomes a module over
$R[x;\sigma]$ under the following scalar product operation:
\par For $m(x)=\sum_{i=0}^n m_ix^i\in M[x;\sigma]$ and $f(x)=\sum_{j=0}^m a_jx^j\in R[x;\sigma]$
$$m(x)f(x)=\sum_{k=0}^{n+m}\parth{\sum_{k=i+j}m_i\sigma^i(a_j)}x^k\eqno(*)$$
Similarly, $M[[x;\sigma]]$ is a module over $R[[x;\sigma]]$. The modules $M[x;\sigma]$ and $M[[x;\sigma]]$ are called
the {\it skew polynomial extension} and the {\it skew power series extension of $M$}, respectively. If $\sigma\in Aut(R)$, then with a scalar product similar to $(*)$ , $M[x,x^{-1};\sigma]$ (respectively, $M[[x,x^{-1};\sigma]]$) becomes a module over $R[x,x^{-1};\sigma]$ (respectively, $R[[x,x^{-1};\sigma]]$). The modules $M[x,x^{-1};\sigma]$ and $M[[x,x^{-1};\sigma]]$ are called the {\it  skew Laurent polynomial extension} and the {\it  skew Laurent power series extension} of $M$, respectively. In \cite{zhang/chen}, a module $M_R$ is called $\sigma$-{\it skew Armendariz}, if $m(x)f(x)=0$ where $m(x)=\sum_{i=0}^nm_ix^i\in M[x;\sigma]$ and $f(x)=\sum_{j=0}^ma_jx^j\in R[x;\sigma]$ implies $m_i\sigma^i(a_j)=0$ for all $i$ and $j$. According to Lee-Zhou \cite{lee/zhou}, $M_R$ is called $\sigma$-{\it Armendariz}, if it is $\sigma$-compatible and $\sigma$-skew Armendariz.

\bigskip

In this paper, we show that if $M_R$ is p.q.-Baer then so is $M[x;\sigma]_{R[x;\sigma]}$ whenever $M_R$ satisfies the condition $\mathcal{C}_2$, and the converse holds when $M_R$ satisfies the condition $\mathcal{C}_1$ (Proposition \ref{prop pqbaer}). Also, if $M_R$ satisfies $\mathcal{C}_2$ and $\sigma$-skew Armendariz, then $M_R$ is a p.p.-module if and only if $M[x;\sigma]_{R[x;\sigma]}$ is a p.p.-module if and only if $M[x,x^{-1};\sigma]_{R[x,x^{-1};\sigma]}$ ($\sigma\in Aut(R)$) is a p.p.-module (Proposition \ref{prop pp}). As a consequence, if $M_R$ is semicommutative and $\sigma$-compatible then:
\NL $M_R$ is a p.p.-module $\Leftrightarrow$ $M_R$ is a p.q.-Baer module $\Leftrightarrow$ $M[x;\sigma]_{R[x;\sigma]}$ is a p.p.-module $\Leftrightarrow$ $M[x;\sigma]_{R[x;\sigma]}$ is a p.q.-Baer module (Theorem \ref{theo2}). Moreover, we obtain a generalization of some results in \cite{baser2007, Baser2, birk/OnpolyExt,lee/zhou}.

\section{Skew polynomials over p.q.-Baer modules}

We start with the next definition.

\begin{Definition}\label{df1}Let $m\in M$ and $a\in R$. We say that $M_R$ satisfies the condition $\mathcal{C}_1$ $($respectively, $\mathcal{C}_2$$)$, if $ma=0$ implies
$m\sigma(a)=0$ $($respectively, $m\sigma(a)=0$ implies $ma=0$$)$.
\end{Definition}

Note that $M_R$ is $\sigma$-compatible if and only if it satisfies $\mathcal{C}_1$ and $\mathcal{C}_2$. Let $M_R$ be a module and $\sigma\in End(R)$.

\begin{Lemma}\label{lemma idempo}If $M_R$ satisfies $\mathcal{C}_1$ or $\mathcal{C}_2$, then $me=m\sigma(e)$
for any $m\in M$ and any $e^2=e\in R$.
\end{Lemma}

\begin{proof} Suppose $\mathcal{C}_2$, from $m\sigma(e)(1-\sigma(e))=0$, we have $0=m\sigma(e)(1-e)=m\sigma(e)-m\sigma(e)e$, so
$m\sigma(e)e=m\sigma(e)$. From $m(1-\sigma(e))\sigma(e)=0$, we have $0=m(1-\sigma(e))e=me-m\sigma(e)e$, so
$m\sigma(e)=m\sigma(e)e=me$. The same for $\mathcal{C}_1$.
\end{proof}

\begin{Proposition}\label{prop pqbaer} Let $M_R$ be a module and $\sigma\in End(R)$.
\NL$(1)$ If $M_R$ is a p.q.-Baer module then so is $M[x;\sigma]_{R[x;\sigma]}$, whenever $M_R$ satisfies the condition
$\mathcal{C}_2$.
\NL$(2)$ If $M[x;\sigma]_{R[x;\sigma]}$ or $M[[x;\sigma]]_{R[[x;\sigma]]}$ is a p.q.-Baer module then so is
$M_R$, whenever $M_R$ satisfies the condition $\mathcal{C}_1$.
\end{Proposition}

\begin{proof}$(1)$ Let $m(x)=m_0+m_1x+\cdots+m_nx^n\in M[x;\sigma]$. Then $r_R(m_iR)=e_iR$, for
some idempotents $e_i\in R\;(0\leq i\leq n)$. Let $e=e_0e_1\cdots e_n$, then $eR=\cap_{i=0}^nr_R(m_iR)$.
We show that $r_{R[x;\sigma]}(m(x)R[x;\sigma])=eR[x;\sigma]$. Let $\phi(x)=a_0+a_1x+a_2x^2+\cdots+a_px^p\in
r_{R[x;\sigma]}(m(x)R[x;\sigma])$. Since $m(x)R\phi(x)=0$, we have $m(x)b\phi(x)=0$ for all $b\in R$. Then
$$m(x)b\phi(x)=\sum_{\ell=0}^{n+p}\parth{\sum_{\ell=i+j}m_i\sigma^i(ba_j)}x^{\ell}=0.$$
\begin{itemize}
\item $\ell=0$ implies $m_0ba_0=0$ then $a_0\in r_R(m_0R)=e_0R$.
\item $\ell=1$ implies $$m_0ba_1+m_1\sigma(ba_0)=0\eqno(1)$$ Let $s\in R$ and take $b=se_0$, so
    $m_0se_0a_1+m_1\sigma(se_0a_0)=0$, since $m_0se_0=0$ we have $m_1\sigma(se_0a_0)=m_1\sigma(sa_0)=0$, so
    $m_1sa_0=0$, thus $a_0\in r_R(m_1R)=e_1R$. In equation $(1)$,
    $m_1\sigma(ba_0)=m_1\sigma(be_1a_0)=m_1\sigma(b)e_1\sigma(a_0)=0$, by Lemma \ref{lemma idempo}. Then equation (1)
    gives $m_0ba_1=0$, so $a_1\in e_0R$.
\item $\ell=2$ implies $$m_0ba_2+m_1\sigma(ba_1)+m_2\sigma^2(ba_0)=0\eqno(2)$$ Let $s\in R$ and take $b=se_0e_1$,
    so $m_0se_0e_1a_2+m_1\sigma(s)e_0e_1\sigma(a_1)+m_2\sigma^2(se_0e_1a_0)=0$, but
    $m_0se_0e_1a_2=m_1\sigma(s)e_0e_1\sigma(a_1)=0$ we have $m_2\sigma^2(se_0e_1a_0)=0$, since $e_0e_1a_0=a_0$ we
    have $m_2\sigma^2(sa_0)=0$ and so $m_2sa_0=0$ for all $s\in R$. Hence $a_0\in e_2R$ (thus, $a_0\in e_0e_1e_2R$).
    Equation $(2)$, becomes $m_0ba_2+m_1\sigma(ba_1)+m_2\sigma^2(b)e_0e_1e_2\sigma^2(a_0)=0$, which gives
    $$m_0ba_2+m_1\sigma(ba_1)=0\eqno(2')$$ Take $b=se_0$ in equation $(2')$, we have $m_0se_0a_2+m_1\sigma(se_0a_1)=0$,
    but $m_0se_0a_2=0$ so $m_1\sigma(se_0a_1)=m_1\sigma(sa_1)=0$ and thus $m_1sa_1=0$, hence $a_1\in e_1R$ (so,
    $a_1\in e_0e_1R$). Equation $(2')$ gives $m_0ba_2=0$, so $a_2\in e_0R$.
\end{itemize}
\par At this point, we have $a_0\in e_0e_1e_2R,\;a_1\in e_1e_2R$ and $a_2\in e_0R$. Continuing this procedure yields
$a_i\in eR$ ($0\leq i\leq n$). Hence $\phi(x)\in eR[x;\sigma]$. Consequently,
$r_{R[x;\sigma]}(m(x)R[x;\sigma])\subseteq eR[x;\sigma]$. Conversely, let $\varphi(x)=b_0+b_1x+b_2x^2+\cdots+b_px^p\in
R[x;\sigma]$. Then
$$m(x)\varphi(x)e=\sum_{\ell=0}^{n+p}\parth{\sum_{\ell=i+j}m_i\sigma^i(b_j)\sigma^{\ell}(e)}x^{\ell}=
\sum_{\ell=0}^{n+p}\parth{\sum_{\ell=i+j}
m_i\sigma^i(b_j)e}x^{\ell}.$$ Since $e\in\bigcap_{i=0}^nr_R(m_iR)$, then $m_iRe=0$ ($0\leq i\leq n$). Thus
$m(x)\varphi(x)e=0$, hence $eR[x;\sigma]\subseteq r_{R[x;\sigma]}(m(x)R[x;\sigma])$. Thus
$r_{R[x;\sigma]}(m(x)R[x;\sigma])=eR[x;\sigma]$, therefore $M[x;\sigma]_{R[x;\sigma]}$ is p.q.-Baer.

\NL$(2)$ Let $0\neq m\in M$. We have $r_{R[x;\sigma]}(mR[x;\sigma])=e{R[x;\sigma]}$ for some idempotent
$e=\sum_{i=0}^ne_ix^i\in R[x;\sigma]$. We have $r_{R[x;\sigma]}(mR[x;\sigma])\cap R=e_0R$. On other hand, we show that
$r_{R[x;\sigma]}(mR[x;\sigma])\cap R=r_R(mR)$. Let $a\in r_R(mR)$ then $mRa=0$, so $mR\sigma^i(a)=0$ for all $i\geq
1$.
So $mR[x;\sigma]a=0$. Therefore $a\in r_{R[x;\sigma]}(mR[x;\sigma])\cap R$. Conversely, let $a\in
r_{R[x;\sigma]}(mR[x;\sigma])\cap R$, then $mR[x;\sigma]a=0$, in particular $mRa=0$, so $a\in r_R(mR)$. Thus $a\in
r_R(mR)=e_0R$, with $e_0^2=e_0\in R$. So $M_R$ is p.q.-Baer. The same method for $M[[x;\sigma]]$.
\end{proof}

\begin{Corollary}[{\cite[Theorem 11]{baser2007}}]\label{cor pqbaer2} $M_R$ is p.q.-Baer
if and only if $M[x]_{R[x]}$ is p.q.-Baer.
\end{Corollary}

\begin{Corollary}[{\cite[Theorem 3.1]{birk/OnpolyExt}}]\label{cor pqbaer1}
$R$ is right p.q.-Baer if and only if $R[x]$ is right p.q.-Baer.
%\NL$(2)$ $R$ is right p.q.-Baer if and only if $R[[x]]$ is right p.q.-Baer.
\end{Corollary}

\bigskip

\ni$M_R$ is called $\sigma$-{\it reduced module} by Lee-Zhou \cite{lee/zhou}, if for any $m\in M$ and $a\in R$:
\NL$(1)$ $ma=0$ implies $mR\cap Ma=0$,
\NL$(2)$ $ma=0$ if and only if $m\sigma(a)=0$.

\begin{Corollary}[{\cite[Theorem 7(1)]{baser2007}}]\label{cor pqbaer3}Let $M_R$ a
$\sigma$-compatible module. Then the following hold:
\NL$(1)$ If $M[x;\sigma]_{R[x;\sigma]}$ is a p.q.-Baer module then so is $M_R$. The converse holds if in addition $M_R$
is $\sigma$-reduced.
\NL$(2)$ If $M[[x;\sigma]]_{R[[x;\sigma]]}$ is a p.q.-Baer module then so is $M_R$.
\end{Corollary}

\begin{Corollary}[{\cite[Corollary 2.6]{Baser2}}]\label{cor pqbaer4}Let $M_R$ be a $\sigma$-compatible module. Then
$M_R$ is p.q.-Baer if and only if $M[x;\sigma]_{R[x;\sigma]}$ is p.q.-Baer.
\end{Corollary}

\section{Skew polynomials over p.p.-modules}

Let $M_R$ be an $\sigma$-Armendariz module, if $me=0$ where $e^2=e\in R$ and $m\in M$, then $mfe=0$ for any $f^2=f\in R$ (by \cite[Lemma 2.10]{lee/zhou}). This result still true if we replace the condition ``$M_R$ is $\sigma$-Armendariz'' by ``$M_R$ is $\sigma$-skew Armendariz satisfying $\mathcal{C}_2$''.

\begin{Proposition}\label{prop pp}Let $M_R$ be a $\sigma$-skew Armendariz module which satisfies the condition $\mathcal{C}_2$. The following statements hold:
\NL$(1)$ $M_R$ is a p.p.-module if and only if $M[x;\sigma]_{R[x;\sigma]}$ is a p.p.-module,
\NL$(2)$ Let $\sigma\in Aut(R)$, then $M_R$ is a p.p.-module if and only if $M[x,x^{-1};\sigma]_{R[x,x^{-1};\sigma]}$ is a p.p.-module.
\end{Proposition}

\begin{proof}$(1)$$(\Leftarrow)$ Is clear by \cite[Theorem 2.11]{lee/zhou}.
\NL $(\Rightarrow)$ Let $m(x)=m_0+m_1x+\cdots+m_nx^n\in M[x;\sigma]$, then $r_R(m_i)=e_iR$, for
some idempotents $e_i\in R\;(0\leq i\leq n)$. Let $e=e_0e_1\cdots e_n$, then $m_ie=0$ for all $0\leq i\leq n$ (\cite[Lemma 2.10]{lee/zhou}) and by Lemma \ref{lemma idempo}, we have $m_i\sigma^j(e)=0$ for all $0\leq i\leq n$ and $j\geq 0$. Therefore $e\in r_{R[x;\sigma]}(m(x))$, so $eR[x;\sigma]\subseteq r_{R[x;\sigma]}(m(x))$. Conversely, let $\phi(x)=a_0+a_1x+\cdots+a_px^p\in r_{R[x;\sigma]}(m(x))$, then $m(x)\phi(x)=0$. Since $M_R$ is $\sigma$-skew Armendariz, we have $m_i\sigma^i(a_j)=0$ for all $i,j$ and with the condition $\mathcal{C}_2$ we have $m_ia_j=0$ for all $i,j$. So $a_j\in r_R(m_i)=e_iR$ for all $i,j$. Thus $a_j\in\cap_{i=0}^n r_R(m_i)=eR$ for each $j$. Then $\phi(x)\in e{R[x;\sigma]}$, therefore $r_{R[x;\sigma]}(m(x))=eR[x;\sigma]$. With the same method, we can prove $(2)$.
\end{proof}

\begin{Corollary}[{\cite[Theorem 11(1a,2a)]{lee/zhou}}]If $M_R$ is $\sigma$-Armendariz. Then:
\NL$(1)$ $M_R$ is a p.p.-module if and only if $M[x;\sigma]_{R[x;\sigma]}$ is a p.p.-module,
\NL$(2)$ Let $\sigma\in Aut(R)$, then $M_R$ is a p.p.-module if and only if $M[x,x^{-1};\sigma]_{R[x,x^{-1};\sigma]}$ is a p.p.-module.
\end{Corollary}

\smallskip

If $M_R$ is a semicommutative module such that, $m\sigma(a)a=0$ implies $m\sigma(a)=0$ for any $m\in M$ and $a\in R$. Then $M_R$ is $\sigma$-semicommutative and hence it satisfies the condition $\mathcal{C}_1$. To see this, suppose that $ma=0$ then $mRa=0$, in particular $mr\sigma(a)a=0$ for all $r\in R$. By the above condition, $mr\sigma(a)=0$ for all $r\in R$. Thus $M_R$ is $\sigma$-semicommutative.

\begin{Lemma}\label{lemma1}If $M_R$ is a semicommutative module such that $m\sigma(a)a=0$ implies $m\sigma(a)=0$ for any $m\in M$ and $a\in R$. Then $M_R$ is $\sigma$-skew Armendariz.
\end{Lemma}

\begin{proof}Let $m(x)=m_0+m_1x+\cdots+m_nx^n\in M[x;\sigma]$ and $f(x)=a_0+a_1x+\cdots+a_px^p\in R[x;\sigma]$.
From $m(x)f(x)=0$, we have $\sum_{i+j=k}m_i\sigma^i(a_j)=0$, for $0\leq k\leq n+p$. So, $m_0a_0=0$.
Assume that $s\geq 0$ and $m_i\sigma^i(a_j)=0$ for all $i,j$ with $i+j\leq s$. Note that for $s+1$, we have
$$m_0a_{s+1}+m_1\sigma(a_s)+\cdots+m_s\sigma^s(a_1)+a_{s+1}\sigma^{s+1}(a_0)=0\eqno(1)$$ Multiplying $(1)$ by
$\sigma^s(a_0)$ from the right hand, we obtain
$$m_0a_{s+1}\sigma^s(a_0)+m_1\sigma(a_s)\sigma^s(a_0)+\cdots+m_s\sigma^s(a_1)\sigma^s(a_0)+a_{s+1}\sigma^{s+1}(a_0)
\sigma^s(a_0)=0,$$ we have $m_0a_0=0$, then $m_0\sigma^s(a_0)=0$ because $M_R$ is $\sigma$-semicommutative, and so $m_0a_{s+1}\sigma^s(a_0)=0$. Also, $m_1\sigma(a_0)=0$ then $m_1\sigma^s(a_0)=0$, thus $m_1\sigma(a_s)\sigma^s(a_0)=0$. Continuing this process until the step $s$, $m_s\sigma^s(a_0)=0$ then $m_s\sigma^s(a_1)$ $\sigma^s(a_0)=0$. Therefore
$m_{s+1}\sigma^{s+1}(a_0)\sigma^s(a_0)=0$. But $$m_{s+1}\sigma^{s+1}(a_0)\sigma^s(a_0)=m_{s+1}\sigma[\sigma^{s}(a_0)]\sigma^s(a_0)=0.$$ So $m_{s+1}\sigma^{s+1}(a_0)=0$ . Therefore, equation $(1)$, becomes
$$m_0a_{s+1}+m_1\sigma(a_s)+\cdots+m_s\sigma^s(a_1)=0\eqno(2)$$ Multiplying $(2)$, by $\sigma^{s-1}(a_1)$ from the
right hand to obtain $m_s\sigma^s(a_1)=0$. Continuing this procedure yields
$$m_0a_{s+1}=m_1\sigma(a_s)=\cdots=m_s\sigma^s(a_1)=a_{s+1}\sigma^{s+1}(a_0)=0.$$ A simple induction shows that
$m_i\sigma^i(a_j)=0$, for all $i,j$.
\end{proof}

\begin{Proposition}\label{prop1}Let $M_R$ be a module such that $m\sigma(a)a=0$ implies $m\sigma(a)=0$ for any $m\in M$ and $a\in R$. If $M_R$ is semicommutative then $M[x;\sigma]_{R[x;\sigma]}$ and  $M[[x;\sigma]]_{R[[x;\sigma]]}$ are semicommutative.
\end{Proposition}

\begin{proof}Let $m(x)=\sum_{i=0}^n m_ix^i\in M[x;\sigma]$, $f(x)=\sum_{j=0}^q a_jx^j\in R[x;\sigma]$ and
$\phi(x)=\sum_{k=0}^p b_kx^k\in R[x;\sigma]$. Suppose that $m(x)f(x)=0$. The coefficients of $m(x)\phi(x)f(x)$ are of
the form
$$\sum_{u+v=w}m_u\sigma^u\parth{\sum_{i+j=v}b_i\sigma^i(a_j)}=\sum_{u+v=w}\parth{\sum_{i+j=v}m_u\sigma^u(b_i)\sigma^{u+i}(a_j)}.$$
By Lemma \ref{lemma1}, $m_u\sigma^u(a_j)=0$, for all $u,j$ and by $\mathcal{C}_1$,
$m_u\sigma^{u+i}(a_j)=0$, for all $i,j,u$. Since $M_R$ is semicommutative then $m_u\sigma^u(b_i)
\sigma^{u+i}(a_j)=0$, therefore $$\sum_{u+v=w}m_u\sigma^u\parth{\sum_{i+j=v}b_i\sigma^i(a_j)}=0.$$ So
$m(x)\phi(x)f(x)=0$, then $M[x;\sigma]_{R[x;\sigma]}$ is semicommutative. The same for $M[[x;\sigma]]_{R[[x;\sigma]]}$.
\end{proof}

According to Baser and Harmanci \cite{baser2007}, a module $M_R$ is {\it reduced} if for any $m\in M$ and $a\in R$, $ma^2=0$ implies $mR\cap Ma=0$. By \cite[Lemma 2.11]{Baser1}, if $M_R$ is semicommutative p.p. or semicommutative p.q.-Baer then it's reduced.

\begin{Corollary}\label{prop2}Let $M_R$ be a semicommutative module satisfying the condition $\mathcal{C}_1$, if $M_R$ is p.q.-Baer or p.p. then $M[x;\sigma]_{R[x;\sigma]}$ and  $M[[x;\sigma]]_{R[[x;\sigma]]}$ are semicommutative.
\end{Corollary}

\begin{proof}Let $a\in R$ and $m\in M$ such that $m\sigma(a)a=0$, then $m(\sigma(a))^2=0$ (by $\mathcal{C}_1$), since $M_R$ is reduced we have $m\sigma(a)=0$. By Proposition \ref{prop1}, $M[x;\sigma]_{R[x;\sigma]}$ and $M[[x;\sigma]]_{R[[x;\sigma]]}$ are semicommutative.
\end{proof}

\begin{Theorem}\label{theo2}If $M_R$ is semicommutative and $\sigma$-compatible. Then the following are equivalent:
\NL$(1)$ $M_R$ is p.p.
\NL$(2)$ $M_R$ is p.q.-Baer,
\NL$(3)$ $M[x;\sigma]_{R[x;\sigma]}$ is p.p.,
\NL$(4)$ $M[x;\sigma]_{R[x;\sigma]}$ is p.q.-Baer,
\end{Theorem}

\begin{proof}$(1)\Leftrightarrow (2)$ By \cite[Proposition 2.7]{Baser1}. $(2)\Leftrightarrow (4)$ By Proposition \ref{prop pqbaer}.
\NL $(3)\Rightarrow (4)$ Since $M_R$ is a p.p.-module, then Corollary \ref{prop2} implies that $M[x;\sigma]_{R[x;\sigma]}$ is semicommutative. Therefore $M[x;\sigma]_{R[x;\sigma]}$ is p.q.-Baer by \cite[Proposition 2.7]{Baser1}. $(4)\Rightarrow (3)$ By Proposition \ref{prop pqbaer}, $M_R$ is p.q.-Baer, since $M_R$ is semicommutative then $M[x;\sigma]_{R[x;\sigma]}$ is semicommutative, and so $M[x;\sigma]_{R[x;\sigma]}$ is a p.p.-module.
\end{proof}

\begin{Corollary}Let $M_R$ be a semicommutative module. Then the following are equivalent:
\NL$(1)$ $M_R$ is p.p.
\NL$(2)$ $M_R$ is p.q.-Baer,
\NL$(3)$ $M[x]_{R[x]}$ is p.p.,
\NL$(4)$ $M[x]_{R[x]}$ is p.q.-Baer,
\end{Corollary}

\begin{Corollary}[{\cite[Theorem 2.8]{Baser2008}}]If $M_R$ is a reduced module. Then the following are equivalent:
\NL$(1)$ $M_R$ is p.p.
\NL$(2)$ $M_R$ is p.q.-Baer,
\NL$(3)$ $M[x]_{R[x]}$ is p.p.,
\NL$(4)$ $M[x]_{R[x]}$ is p.q.-Baer,
\end{Corollary}

\begin{proof}Every reduced module is semicommutative by \cite[Lemma 1.2]{lee/zhou}.
\end{proof}

%{\bf Received: October 25, 2010}

\end{document}